\documentclass[12pt]{article}
\usepackage{latexsym,amsmath,amssymb}
\usepackage{graphicx} 
 \usepackage{float}%
\usepackage{color}
\definecolor{DarkBlue}{rgb}{0.1,0.1,0.5}
\definecolor{Red}{rgb}{0.9,0.0,0.1}

\newtheorem{theorem}{Theorem}[section]

\newtheorem{lemma}[theorem]{Lemma}
\newtheorem{definition}[theorem]{Definition}

\newtheorem{remark}[theorem]{Remark}

\def \BM{{\Bbb M}}

\def \BR{{\Bbb R}}

\def \BZ{{\Bbb Z}}

\begin{document}
\title{Change point detection based on method of moment estimators}
\author{
Ilia Negri\footnote{Corresponding author,  e-mail: {\tt ilia.negri@unibg.it}. Department of Economic Science, University of Bergamo. Viale dei Caniana, 2, 24127 Bergamo, Italy.} \ \ and \ \ Yoichi Nishiyama\footnote{Faculty of International Research and Education, Waseda University. 1-6-1 Nishi-Waseda, Shinjuku-ku, Tokyo 169-8050, Japan.}
}
\maketitle


\begin{abstract}
A change point detection procedure using the method of moment estimators is proposed. 
The test statistics is based on a suitable $Z$-process. The asymptotic behavior of this process is established under both the null and the alternative hypothesis and the consistency of the test is also proved. An estimator for  the change point is proposed and its consistency is derived. Some examples of  this  method applied to a parametric family of random variables are presented. 

\end{abstract}




\section{Introduction}\label{introduction}

An important issue in statistics concerns to test on structural change problems.   
This problem arises naturally in quality control context, where one is faced about the output of a production line and 
would find any departure from an acceptable standard of the production. 
From the statistical point of view, the problem consists in testing whether there is a statistically significant change point in a sequence of chronologically ordered data. 
The problem for an i.i.d.\ sample was first considered in the paper of Page \cite{Page-55}, see also Hinkley \cite{Hinkley-70},  
and, for a general survey of the change point detection and estimation,  see Chen and Gupta \cite{Chen-Gupta-01}. 
The parameter change point problem  became very popular in regression and time series models. 
This is because  these models can be used to describe structural changes that often occur in financial and economic phenomena (due for example to a change of political situation or to a change of economic policy) or in environmental phenomena (due to sudden changes in weather situation or in the case of a natural catastrophe). For regression models see, for example, Hinkley \cite{Hinkley-69}, Quandt \cite{Quandt-60}, Brown {\em et al.} \cite{Brown-et-el-75}, Chen \cite{Chen-98}. For time series models, see for example, Picard \cite{Picard-85}. Ling \cite{Ling-07} and  Lee {\em et al.} \cite{Lee_et_al}.
For a general review of parametric methods and analysis, refer also to Cs{\"o}rg{\H o} and Horv{\'a}th \cite{Csorgo-H-97} and to Chen and Gupta \cite{Chen-G-00}.

The aim of this  paper is to show  a very simple  approach for change point detection, using 
the partial sum process  based on the method of moment estimator. This idea naturally turns out from  a general approach to change point problems, developed in Negri and Nishiyama  \cite{NN17} based on the partial sum processes of estimating equations, called the Z-process method. To understand how the idea arises, let us recall some results about change point problem. 
The partial sum process  is defined as
\[
\BM_n(u,\theta)=\frac{1}{n}\sum_{k=1}^{[un]} \log f(X_k;\theta), \quad \forall u \in [0,1], 
\]
where
$ f(\cdot;\theta) $ is a parametric family of probability densities 
with respect to a measure  $ \mu $, defined on a suitable measurable space,  $ ({\cal X},{\cal A},\mu) $ and  $ \theta \in \Theta \subset \BR^d $.  Let $ X_1,X_2, \ldots $ be an independent sequence of $ {\cal X} $-valued random variables from this parametric model. Introduce the gradient vectors $ \BZ_n(u,\theta)=\dot{\BM}_n(u,\theta) $. 
Let $ \tilde{\theta}_n $ be the maximum likelihood estimator (MLE) for the full data $ X_1,\ldots,X_n $. 
The MLE is  a special case of $Z$-estimators, that is, $ \tilde{\theta}_n $ is the solution to the estimating equation 
\[
\BZ_n(1,\theta)=
\dot{\BM}_n(1,\theta)=0. 
\]
To detect if there is a change point, 
Horv\'{a}th and Parzen \cite{Hor-P-94} are apparently the firsts to introduce the test statistics based on the Fisher-score process
\begin{equation}
{\cal F}_n= n \sup_{u \in [0,1]}\left | \BZ_n(u,\tilde{\theta}_n)^\top \widehat{I}_n^{-1}\BZ_n(u,\tilde{\theta}_n) \right |, 
\label{basic:eq}
\end{equation}
where $ \widehat{I}_n $ is a consistent estimator for the Fisher Information matrix $ I(\theta_0) $.  It is straightforward  that 
\[
{\cal F}_n \to^d \sup_{u \in [0,1]}||B^\circ(u)||^2. 
\]
in the Skorohod space $ D[0,1] $, where $ u \leadsto B^\circ(u) $ 
is a vector of independent standard Brownian bridges and $\to^d $ denotes the convergence in distribution in Skorohod space with respect the {\em sup} norm. 
This was proved by  Horv\'{a}th and Parzen \cite{Hor-P-94}, although they didn't discuss the asymptotic behavior of the test 
under the alternative.  Negri and Nishiyama \cite{NN12}  took the same approach to the change point problem 
for an ergodic diffusion process model based on the continuous observation and 
proved also the consistency of the test under an alternative which has sufficient generality. The general approach to change point problem by Negri and Nishiyama \cite{NN17} is not just a simple generalization of Fisher-score process method in the case of independent random sequences proposed by Horv\'{a}th and Parzen\cite{Hor-P-94}, but  in their framework it is possible to treat  new applications in broad spectrum of statistical change point problems including not only models for ergodic dependent data but also non-ergodic cases. Moreover, as it is very important in discussing any kind of statistical testing hypotheses problems, to study the behavior of the test statistics under certain alternatives, an argument to prove the consistency of the test based on the proposed method under some specified alternatives is also developed. 

This paper starts from the consideration that in the test statistics based on the Fisher-score process \eqref{basic:eq} and its generalization, the $Z$-process can be considered also when the estimating equations are introduced to define the method of moment estimator (MME). This estimator is  the solution of a particular estimating equation and so the $Z$-process arise naturally without be a gradient of a $M$-process as in the case of the MLE. 
The MME has the advantage to be defined under very mild conditions and to be  very simple to compute.

The main result of this paper is to prove that the test statistics based on the $Z$-process that generates the MME converges  to the supremum of a Brownian Bridge. Moreover the asymptotic properties under the alternative are easily deduced. The form of the test statistics suggest also  a  procedure not only to establish  if there is a change point, but if it is the case, where or when this change point appears. The estimator of the change point is introduced as the instant where the test statistics attains its maximum. The consistency of this estimator is also proved. 

The rest of the paper is organized as follows. In the next Section some general notations are given. In Section \ref{MME:sec} the MME estimator is defined. 
The change point problems is  presented in Section \ref{CPP:sec}. In particular in Subsection  \ref{sub1}  the asymptotic properties under the null hypothesis are proved. In Subsection  \ref{sub2}  the asymptotic properties under the alternative are considered. The estimation of the change point and the asymptotic behavior of the estimator are presented in Subsection \ref{sub3}. Finally in Section \ref{sim:sect} a simulation study on the Gamma distribution is presented.

\section{Notation}
Let $ D[0,1] $ be the space of functions defined on $ [0,1] $ 
taking values in a finite-dimensional Euclidean space, 
which are right continuous and have left hand limits; 
we equip this space with the Skorohod metric. 
Throughout this paper, all random processes, 
denoted as  $ u \leadsto X(u) $, are assumed to take values in $ D[0,1] $. 
See for example Kallemberg \cite{Kal02} for these definitions. 

In what follows, the parametric space $ \Theta $ is a bounded, open, convex subset of $ \BR^d $, 
where $ d $ is a fixed positive integer. 
The word ``vector'' always means ``$ d $-dimensional real column vector'', 
and the word ``matrix'' means ``$ d \times d $ real matrix''. 
The Euclidean norm is denoted by $ ||v||:=\sqrt{\sum_{i=1}^d |v^{(i)}|^2} $ 
for a vector $v$ where $ v^{(i)} $ denotes the $ i $-th component of $v $, 
and by $ ||A||:=\sqrt{\sum_{i,j=1}^d|A^{(i,j)}|^2} $ for a matrix $A$ 
where $ A^{(i,j)} $ denotes the $ (i,j) $-component of $ A $. 
Note that $ ||Av|| \leq ||A||\cdot ||v|| $ and $ ||AB|| \leq ||A||\cdot||B|| $ for 
vector $ v $ and matrices $ A, B $. 
The notations $ v^\top $ and $ A^\top $ denote the transpose. 
The notations $ \to^P $ and $ \to^d $ mean the convergence in probability with respect to a probability measure $P$ and 
the convergence in distribution, as $ n \to \infty $, respectively.

\section{Method of moments estimator}\label{MME:sec}
Let us recall, following Van der Vaart \cite{Vaa-98}, the MME. Some notations we will use along all the paper are introduced in this section. 
The method of moments estimator gives the estimate by comparing functionals of sample and  their theoretical moments. In general it can be view as a $Z$-estimator. 
Let $X,  X_1,X_2,\ldots $ be an i.i.d.\ sample from a distribution $ P_\theta $ on a measurable space $ ({\cal X},{\cal A}, \mu)$. Suppose that the parameter  
$\theta\in \Theta \subset \BR^d$. 
Let $\psi : {\cal X} \to   \BR^d$ be a measurable function on $ {\cal X} $. We write $\psi=( \psi^{(1)},\ldots,\psi^{(d)})^\top $.   Let's introduce the $d$-dimensional vector
$$
e(\theta)=E_\theta (\psi(X)) =( E_\theta[\psi^{(1)}(X)], \ldots , E_\theta[\psi^{(d)}(X)])^\top,
$$
where $E_\theta$ is the expected value with respect to$ P_\theta $,  the $d\times d$ symmetric matrix
$$
\Sigma(\theta)= Cov_\theta (\psi(X))=E_\theta \left( (\psi(X)-e(\theta) ) (\psi(X)-e(\theta))^\top\right),
$$
and the the $d\times d$ derivative matrix
$$
V(\theta)=\dot e(\theta)=
\begin{bmatrix}
\frac{\partial}{\partial \theta_{j}}e^{(i)}(\theta)
\end{bmatrix}
\quad i,j=1,\ldots,d.
$$
Define 
$$
\BZ_n(\theta)=\frac{1}{n}\sum_{k=1}^n
(\psi^{(1)}(X_k)-e^{(1)}(\theta), \ldots, \psi^{(d)}(X_k)-e^{(d)}(\theta))^\top= \frac{1}{n}\sum_{k=1}^n (\psi (X_k)-e(\theta)).
$$
The solution $\hat \theta_n$ of the system of equations $\BZ_n(\hat \theta_n)=0$ is called method of moments estimator.  If $e$ is one-to-one then the moment estimator is uniquely determined as $\hat \theta_n =e^{-1}(\frac{1}{n}\sum_{k=1}^n (\psi (X_k)))$. 
See Van der Vaart \cite{Vaa-98} for more details and properties of MME's.

\section{Change point problem}
\label{CPP:sec}
Let $ ({\cal X},{\cal A},\mu) $ be a measurable space, 
and let a parametric family of probability densities $ f(\cdot;\theta) $ 
with respect to $ \mu $, where $ \theta \in \Theta \subset \BR^d $, be given. 
Let $ X_1,X_2, \ldots $ be an independent sequence of $ {\cal X} $-valued random variables from this parametric model. 

We consider the following testing problem (change point problem):
\begin{center}
$ H_0 $: {\em the true value $ \theta_0 \in \Theta $ 
does not change during $ u \in [0,1] $}; 
\end{center}
versus any alternative that for the moment we can state as 
$ H_1 $: {\em there is a change in some $u\in (0,1)$}. 

Let us denote with $P_{\theta_0}$ the probability measure under $H_0$. 
To deal with this change point problem let us introduce the partial sum process  
\[
\BZ_n(u,\theta)=\frac{1}{n} \sum_{k=1}^{[un]} (\psi (X_k)-e(\theta)), \quad \forall u \in [0,1].
\]
The method of moments estimator is the solution to $\BZ_n(1,\theta)=0$.

Define the
derivative matrix as $ \dot{\BZ}_n(u, \theta)=\{ \dot{\BZ}_n^{(i,j)}(u, \theta) \}_{(i,j) \in \{ 1,\ldots,d \}^2} $, 
where $ \dot{\BZ}_n^{(i,j)}(u, \theta)=\frac{\partial}{\partial \theta_j}\BZ_n^{(i)}(u, \theta) = -\frac{\partial}{\partial \theta_{j}}e^{(i)}(\theta)$.


\subsection{Asymptotic properties under the null hypothesis}
In this section the asymptotic distribution under the null hypothesis of the test statistic is proved. The main result is given by Theorem \ref{Null_main}. Before to state it some preliminary results are proved. 
\label{sub1}
\begin{theorem}
Let us suppose that $V(\theta_0)$ is invertible and that $E_{\theta_0}[||\psi(X)||^{2}]<\infty$. Then 
$$
\sqrt{n}(\hat \theta_n -\theta_0) = - V(\theta_0)^{-1}\sqrt{n}\BZ_n(1,\theta_0) +o_{P_{\theta_0}}(1).
$$
Moreover
$$
\sqrt{n}(\hat \theta_n -\theta_0) \to^{d} - V(\theta_0)^{-1}\Sigma^{1/2}(\theta_0) \xi
$$
where $\xi$ is a $d$-dimensional  standard normal random variable. 
\end{theorem}
\begin{proof}
The first line can be proved as in Van der Vaart \cite{Vaa-98}. The second line is proved as a special case of Lemma \ref{Lemma1} below. 
\end{proof}

\begin{lemma}
\label{Lemma1}
Let $\{B(u)\}_{u\in [0,1]}$ be a $d$-dimensional standard Brownian motion. It holds
$$
\sqrt{n} \BZ_n(u,\theta_0) \to^{d} \Sigma^{1/2}(\theta_0) B(u)
$$ 
in $D[0,1]$. 
\end{lemma}
\begin{proof}
The result follows from Donsker's theorem. 
\end{proof}
\begin{lemma}
\label{Lemma2}
Let $\{B(u)\}_{u\in [0,1]}$ be a standard Brownian motion in $\BR^d$. It holds
$$
\sqrt{n} \BZ_n(u,\hat \theta_n) \to^{d} \Sigma^{1/2}(\theta_0) \left(B(u) -uB(1)\right)
$$ 
in $D[0,1]$. 
\end{lemma}
\begin{proof}
It follows from the Taylor expansion that 
\begin{eqnarray*}
\lefteqn{
\sqrt{n} \BZ_n(u,\widehat{\theta}_n)
}\\
&=&
\sqrt{n} \BZ_n(u,\theta_0)
+\dot{\BZ}_n(u,\widetilde{\theta}_n(u))\sqrt{n} (\widehat{\theta}_n-\theta_0)
\\
&=&
\sqrt{n} \BZ_n(u,\theta_0)
+uV(\theta_0)(-V(\theta_0)^{-1}\sqrt{n} \BZ_n(1,\theta_0)
+o_{P_{\theta_0}}(1)
\\
&=& 
\sqrt{n} \left(  \BZ_n(u,\theta_0) - u  \BZ_n(1,\theta_0)\right) +o_{P_{\theta_0}}(1).
\end{eqnarray*} 
The term $ \widetilde{\theta}_n(u) $ appearing above  is a random vector 
on the segment connecting $ \theta_0 $ and $ \widehat{\theta}_n $.
Now it follows  from Lemma \ref{Lemma1} that
$$
\sqrt{n} \left( \BZ_n(u,\theta_0) - u  \BZ_n(1,\theta_0)\right)  \to^{d} 
 \Sigma^{1/2}(\theta_0) \left(B(u) -uB(1)\right) \quad \mbox{in } D[0,1].
$$
\end{proof}
\begin{lemma}
\label{Lemma_sigma}
Let define 
$$
\widehat{\Sigma}_n= \frac{1}{n}\sum_{k=1}^n  (\psi (X_k)-e(\widehat {\theta}_n)) (\psi (X_k)-e(\widehat {\theta}_n))^\top.
$$
It holds that $\widehat{\Sigma}_n \to^{P_{\theta_0}}  \Sigma(\theta_0)$. 
\end{lemma}
\begin{proof}
The result follows by the strong law of large number observing that $e(\widehat {\theta}_n)=\frac{1}{n}\sum_{k=1}^n \psi (X_k)$. 
\end{proof}

To test if there is a change in the parameter value, let us introduce the statistic 
\begin{equation}
{\cal T}_n= n \sup_{u \in [0,1]} \left | \BZ_n(u,\widehat{\theta}_n)^\top\widehat{\Sigma}_n^{-1}\BZ_n(u,\widehat{\theta}_n)\right |.
\label{test_stat_MM}
\end{equation}
This test statistics has the form very similar to \eqref{basic:eq}. In  \eqref{test_stat_MM} the MME appears instead of the MLE, the  matrix $\widehat{\Sigma}_n$ plays the role of the  consistent estimator of the Fisher Information matrix. Moreover the $Z$-process involved have different form.  The asymptotic distribution is in any case the supremum of the norm of a vector of Brownian Bridge. 

\begin{theorem}
\label{Null_main}
Let the conditions of the above Lemmas hold true. Then 
$$
\sup_{u \in [0,1]} | n\BZ_n(u,\widehat{\theta}_n)^\top\widehat{\Sigma}_n^{-1}\BZ_n(u,\widehat{\theta}_n)| \to^{d}  \sup_{u \in [0,1]} \|B(u)-uB(1)\|. 
$$
\end{theorem}
\begin{proof}
The claim follows from Lemma \ref{Lemma2} and the continuous mapping theorem. 
\end{proof}

\subsection{Asymptotic properties under the alternative} 
\label{sub2}
Let us suppose that under the alternative hypothesis there exists a certain instant $n_*$ where the value of the parameter changes. More precisely in the most typical form the alternative in the change point problems can be defined as follow.

\begin{center}
$ H_1 $: {\em there exists a constant $ u_* \in (0,1) $ such that 
the true value is $ \theta_0 \in \Theta $ for $ u \in [0,u_*] $, 
and $ \theta_1 \in \Theta $ for $ u \in (u_*,1] $, 
where $ \theta_0 \not= \theta_1 $}. 
\end{center}

\begin{lemma}
Under $H_1$,  $\hat \theta_n$ converges in probability to $\theta_*$, where $\theta_*$ is the solution of 
 the following equation 
$$
u_*e(\theta_0) + (1-u_*)e(\theta_1) - e(\theta) = 0.
$$
\end{lemma}
\begin{proof}
The result follows by Theorem 5.7 and 5.9 of Van der Vaart \cite{Vaa-98}.

\end{proof}
\begin{theorem}
Let $u_*$, $\theta_0$ and $\theta_1$ as defined in $H_1$. Let us define $\Sigma_*=u_*\Sigma(\theta_0)+ (1-u_*)\Sigma(\theta_1)$. Let $\lambda_*$ the smallest eigenvalue of 
$\Sigma_*^{-1}$.  Let us assume that $\lambda_*>0$ and $\|e(\theta_0)-e(\theta_1)\|>0$. Then the test based on the statistics 
$
{\cal T}_n$ is consistent.
\end{theorem}
\begin{proof}
We can prove as in Lemma \ref{Lemma_sigma} that under $H_1$ it holds $\widehat{\Sigma}_n \to^{P_*}  \Sigma_*$. Here $P_*$ is the probability measure corresponding to  $ P_{\theta_0}$ before the change and corresponding to $ P_{\theta_1}$ after the change point. 
It can be proved that under $H_1$ it holds
$$
\BZ_n(u_*,\widehat{\theta}_n) \to^{P_*} u_*(1-u_*)(e(\theta_0- e(\theta_1)).
$$
Hence 
\begin{eqnarray*}
{\cal T}_n &\geq & n\left( u_*^2(1-u_*)^2(e(\theta_0- e(\theta_1))^\top \Sigma_*^{-1} (e(\theta_0- e(\theta_1))+o_{P_*}(1)\right) \\
&\geq & n\left( u_*^2(1-u_*)^2 \lambda_* \| e(\theta_0- e(\theta_1)\|^2  +o_{P_*}(1)\right).
\end{eqnarray*}
This complete the proof.
\end{proof}
\subsection{Estimation of the change point}
\label{sub3}
As an estimator of the change point, let us  define

$$
\hat u_n = \arg\max {\cal T}_n(u)
$$
where
$ {\cal T}_n(u)=| n\BZ_n(u,\widehat{\theta}_n)^\top\widehat{\Sigma}_n^{-1}\BZ_n(u,\widehat{\theta}_n)|$.

Let us introduce the following function.  
\begin{definition}
Let $u\in [0,1]$ and $\theta \in \Theta$. Let us define the following function
\begin{equation}
{\cal Z}(u, \theta)= 
\begin{cases}
u(e(\theta_0)- e(\theta)), & \text{ if } u\leq u_*,\\
u_*(e(\theta_0)- e(\theta_1)) + u(e(\theta_1)- e(\theta)), & \text{ if } u >u_*.\\
\end{cases}
\end{equation}
\end{definition}
Note that for $\theta=\theta_*$ we have
$$
{\cal Z}(u, \theta_*)= 
\begin{cases}
u(1-u_*)(e(\theta_0)- e(\theta_1)), & \text{ if } u\leq u_*,\\
u_*(1-u)(e(\theta_0)- e(\theta_1)),  & \text{ if } u >u_*.\\
\end{cases}
$$
\begin{theorem}
Under $H_1$, $\hat u_n$ converges in probability to $u_*$. 
\end{theorem}
\begin{proof}
Observe that 
$$
\sup_{u, \theta} \left | \BZ_n(u, \theta) -{\cal Z}(u,\theta)\right | \to^{P_*} 0.
$$
Remembering that $\hat\theta_n\to^{P_*} \theta_*$ by Lemma 3.6 
we have
$$
\sup_{u} \left | \BZ_n(u,  \hat\theta_n) -{\cal Z}(u,\theta_*)\right | \to^{P_*} 0.
$$
Since the condition
$$
\sup_u \left| \frac{{\cal T}_n(u) }{n} -{\cal Z}(u, \theta_*)^T \Sigma_*^{-1} {\cal Z}(u,\theta_*) \right| \to ^{d} 0
$$
 holds, and  
 the maximizer of $u \mapsto {\cal Z}(u, \theta_*)^T \Sigma_*^{-1} {\cal Z}(u,\theta_*)$ is $u_*$,
the result follows by Corollary 3.2.3 of Van der Vaart and Wellner \cite{Vaa-W-96}. 
\end{proof}

\section{Example and simulation study}
\label{sim:sect}
We apply our method to test if there is change point in a Gamma model. 
Let  $X, X_1, \ldots, X_n$ be i.i.d. Gamma random variables. For  $\theta=(\alpha, \lambda)^T$, $\alpha>0$, $\lambda>0$, the density is given by
$$
f(x; \theta )= \frac{\lambda^{\alpha}}{\Gamma(\alpha)}x^{\alpha-1} e^{- \lambda x }, \quad \forall x>0,
$$
where $\Gamma (\alpha )=\int_{0}^{\infty }t^{\alpha -1}e^{-t}dt$. Define $\psi :  \BR \to   \BR^2$ by $\psi(x)=(x, x^2)^\top$.  
We have
$$
e(\theta) =\begin{bmatrix}
E_\theta(X)\\
E_\theta(X^2)\\
\end{bmatrix}= 
\begin{bmatrix}
 \frac{\alpha}{\lambda}\\
 \frac{\alpha(\alpha+1)}{\lambda^2}\\
 \end{bmatrix}.
 $$
 Define $\bar X_n= \frac{1}{n} \sum_{j=1}^n X_j$ and $\bar X_n^2= \frac{1}{n} \sum_{j=1}^n X_j^2$. The moment estimator is given by  
$$
\hat \theta_n=
\begin{bmatrix}
\hat \alpha_n\\
\hat \lambda_n\\
\end{bmatrix}=
\begin{bmatrix}
\frac{(\bar X_n)^2}{\bar X_n^2-(\bar X_n)^2}\\
 \frac{\hat \alpha_n}{\bar X_n}
\end{bmatrix}.
$$
When  the null hypothesis is rejected, we estimate the chance point with the proposed statistics $\hat u_n$. 
We simulate $n$ values  of a Gamma distribution, for different values of $n$, to asses the asymptotic results we have presented in the previous Sections. 

First of all, we estimate the empirical size of the test under the Null hypothesis for different values of the parameters of the Gamma distribution.  Then, we compute the empirical power of the test under the alternative that there is a change point in the parameter. Finally we study the consistency of the change point estimator. 

The set up of the design of our simulation study is the following. The number of the Monte Carlo experiments is $m=10000$.
In the simulation study we set the change point in three different points, respectively $u_*= 0.50, 0.75, 0.90$. Moreover we choose different values of the parameters for the change point. 
To study the asymptotic behavior of the test statistics under the Null and the Alternative Hypothesis and to study the consistency of the estimator of the change point estimator we set three different values of the number of observations, respectively $n=50, 100, 500$.    

The level of the test  is fixed at $\varepsilon=0.05$. The critical values of the test statistics have been reported in  Lee {\em et al.} \cite{Lee_et_al}. For the given level the critical value is 
$c_\varepsilon=2.408$.  
\subsubsection*{Results}
The empirical size of the test is  less and closed to the theoretical value for any $n$ and any choice of the parameters.  See Table 1.

\begin{table}[ht]
\centering
\begin{tabular}{ll|rrr}
  \hline
  & & $n=50$ &  $n=100$&  $n=500$ \\ 
  \hline
$\alpha=1$& $\lambda=1$ & 0.0233 & 0.0276 & 0.0429 \\ 
$\alpha=1$&$\lambda=0.01$ & 0.0217 & 0.0291 & 0.0374 \\ 
 $\alpha=2$& $\lambda=1$ & 0.0274 & 0.0307 & 0.0432 \\ 
   \hline
\end{tabular}
\caption{Empirical size of the test for different values of the parameters and different number of observations $n$.  }

\end{table}

Regarding the empirical power, it increases and it reaches 1 as $n$ increases. For example in the case reported in Table 2, the test reaches empirical power equal 1 for $n=100$  when the change point is settled at $u_*=0.50$. The empirical power is 0.995 for $u_*=0.75$ and only 0.170 for   $u_*=0.90$. Anyway in the worst case $u_*=0.90$ the empirical power reaches 1 for $n=500$ as the asymptotic result suggest. 

\begin{table}[ht]
\centering
\begin{tabular}{l|lll}
  \hline
$u_*$ & $n=50$ &  $n=100$&  $n=500$ \\ 
  \hline
$0.50$ & 0.9795 & 1 & 1 \\ 
  $0.75$ & 0.6813 & 0.9953 & 1 \\ 
  $0.90$ & 0.0650 & 0.1696 & 1 \\ 
   \hline
\end{tabular}
\caption{Empirical power of the test for different values of $n$. The parameter $\alpha=1$ while the parameter $\lambda$ change from 0.01 to 0.05 at the point $u_*$.  }

\end{table}
Moreover the empirical power convergence to 1 is reached for lower values of $n$ as the distance between the parameter increase. This  can be seen in Tables 3 and 4, where the parameter $\alpha$ change  from 1 to 2 and 4 respectively. 

\begin{table}[ht]
\centering
\begin{tabular}{l|lll}
  \hline
$u_*$  & $n=50$ &  $n=100$&  $n=500$ \\ 
  \hline
$0.50$ & 0.67 & 0.96 & 1 \\ 
 $0.75$ & 0.68 & 0.99 & 1 \\ 
$0.90$ & 0.07 & 0.17 & 1 \\ 
   \hline
\end{tabular}
\caption{Empirical power of the test for different values of $n$. The parameter $\lambda=1$ while the parameter $\alpha$ change from 1 to 2 at the point $u_*$.  }
\end{table}

\begin{table}[ht]
\centering
\begin{tabular}{l|lll}
  \hline
 $u_*$ & $n=50$ &  $n=100$&  $n=500$ \\ 
  \hline
$0.50$ & 1 & 1 & 1 \\ 
 $0.75$ & 0.69 & 1& 1 \\ 
  $0.90$ & 0.06 & 0.17 & 1 \\ 
   \hline
\end{tabular}
\caption{Empirical power of the test for different values of $n$. The parameter $\lambda=1$ while the parameter $\alpha$ change from 1 to 4 at the point $u_*$.  }
\end{table}

As discussed in Section \ref{sub3} we propose an estimator for the change point $u_*$. 
The estimator $\hat u_n$ is computed for any trajectory as the argument where the test statistics attains its maximum. In the Figure \ref{fig_1} the histogram of the 10000 values of $\hat u_n$ is plotted for $n=500$ for the change point reported reported in Table 2, where the parameter $\alpha$ remain unchanged whereas the parameter $\lambda$ change from $0.01$ to $0.05$. 

\begin{figure}[h]
\label{fig_1}
\includegraphics[height=0.40\textheight,width=\textwidth]{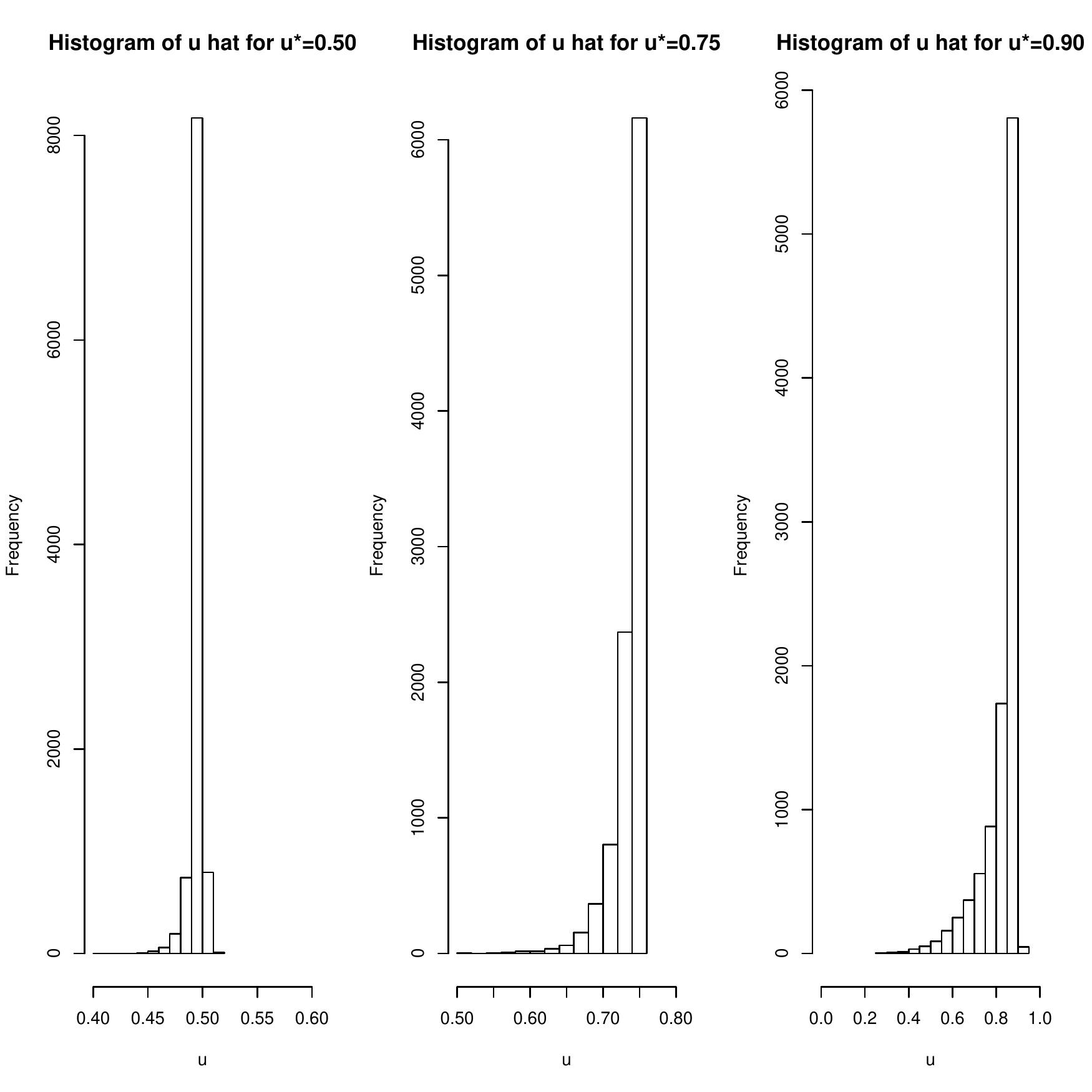}
\caption{Histogram of 10000 values of $\hat u_n$ for $n=500$ for different values of $u_*$}
\end{figure}

This is a case where the empirical power is 1, so the change point is always detected and it is correctly estimated. Indeed, it can be seen that for the three values of $u_*$ in each histogram its maximum is attained in correspondence of the class that contain the true value $u_*$. 
In Table 5, the mean, the standard deviation (sd) and the estimated value for the root mean square error (RMSE) for the estimator $\hat u_n$ are computed for the three different values of $u_*$ when $n=500$,  the parameter $\alpha=1$ while the parameter $\lambda$ changes from 0.01 to 0.05 at the point $u_*$.

\begin{table}[h]
\centering
\begin{tabular}{l|llc}
  \hline
 $u_*$ & mean $\hat u_n$ & sd $\hat u_n$ & est.  RMSE $\hat u_n$ \\ 
  \hline
0.50 & 0.497 & 0.006 & 0.007 \\ 
  0.75 & 0.737 & 0.022 & 0.025 \\ 
  0.90 & 0.831 & 0.090 & 0.114 \\ 
   \hline
\end{tabular}
\caption{True value $u_*$, mean, standard deviation (sd) and estimated value for the root mean square error (RMSE) for the estimator $\hat u_n$. Here $n=500$,  the parameter $\alpha=1$ while the parameter $\lambda$ changes from 0.01 to 0.05 at the point $u_*$. 
}
\end{table}

The mean value is closer to the true value as far as the change point is closer to 0.50. Moreover the variability (measured with the standard deviation) increases as the change point approaches the end of the period of observation. The values of the estimated RMSE closed to the standard deviation are an evidence that the estimator is unbiased. 

All the simulation results  presented in this section are consistent with the theoretical results. We can conclude the our procedure is able to establish if there is a change point in our model and when the test reject the null hypothesis the change point can be  estimate without big error.

\vskip 20pt
\par\noindent
{\bf Acknowledgements.} 
This work was supported by the Department of Management,
Information and Production Engineering, Grant 2017 (I.N.) and by Grant-in-Aid for Scientific Research (C), 15K00062, 18K11203, from Japan Society for the Promotion of Science (Y.N.).

\end{document}